\documentclass[12pt]{amsart}
\usepackage[T1]{fontenc}
\usepackage[latin9]{inputenc}
\usepackage{geometry}
\usepackage[english]{babel}
\usepackage{float}
\usepackage{amsthm}
\usepackage{amstext}
\usepackage{amssymb}
\usepackage{graphicx}
\usepackage{esint}
\usepackage{hyperref}

\makeatletter
\numberwithin{equation}{section}
\numberwithin{figure}{section}
\theoremstyle{plain}
\newtheorem{thm}{\protect\theoremname}
  \theoremstyle{remark}
  \newtheorem{rem}[thm]{\protect\remarkname}
  \theoremstyle{remark}
  \newtheorem{example}[thm]{\protect\examplename}
  \theoremstyle{definition}
  \newtheorem{defn}[thm]{\protect\definitionname}
  \theoremstyle{plain}
  \newtheorem{prop}[thm]{\protect\propositionname}
  \theoremstyle{plain}
  \newtheorem{lem}[thm]{\protect\lemmaname}
   \newtheorem{cor}[thm]{\protect\corname}

\def\e{{\epsilon}}
\def\G{{\mathcal{G}}}
\def\g{{\mathfrak{g}}}
\def\Ts{{T^*}}
\def\R{\mathbb{R}}

\def\id{\operatorname{id}}

\def\sym{\operatorname{sym}}

\def\e{\epsilon}
\def\t{\tau}
\def\R{\mathbb{R}}

\makeatother

  \providecommand{\definitionname}{Definition}
  \providecommand{\examplename}{Example}
  \providecommand{\lemmaname}{Lemma}
  \providecommand{\propositionname}{Proposition}
  \providecommand{\remarkname}{Remark}
\providecommand{\theoremname}{Theorem}
\providecommand{\corname}{Corollary}

\begin{document}
\selectlanguage{english}

\title{formal symplectic realizations}

\author{Alejandro Cabrera and Benoit Dherin}

\address{A. C.: Departamento de Matematica Aplicada, Instituto de Matematica\\
Universidade Federal do Rio de Janeiro, CEP 21941-909, Rio de Janeiro - RJ \\ Brazil.}
\email{acabrera@labma.ufrj.br}

\address{B.D.: Department of Mathematics, University of California,
Berkeley, CA 94720-3840, USA}
\email{dherin@math.berkeley.edu}

\begin{abstract}
We study the relationship between several constructions of symplectic
realizations of a given Poisson manifold. Our main result is a general formula for a formal symplectic realization in the case of an arbitrary Poisson structure on
$\R^n$. This formula is expressed in terms of rooted trees and elementary differentials, building on the work of Butcher, and the coefficients are shown to be a generalization of Bernoulli numbers appearing in the linear Poisson case. We also show that this realization coincides with  a formal version of the original construction of Weinstein, when suitably put in global Darboux form, and with the realization coming from tree-level part of Kontsevich's star product. We provide a simple iterated integral expression for the relevant coefficients and show that they coincide with underlying Kontsevich weights.
\end{abstract}
\maketitle
\tableofcontents{}

\section{Introduction}

A symplectic realization of a Poisson manifold $(M,\pi)$ is a Poisson
map from a symplectic manifold $(S,\omega)$ to $(M,\pi)$, which
is also a surjective submersion. Symplectic realizations are very
natural objects in Poisson geometry from the point of view of integration
and quantization theory of Poisson manifolds.

Namely, following the standard ``quantization dictionary'' that
associates a quantum algebra $\mathcal{A}_{\pi}$ of observables with
a Poisson manifold $(M,\pi)$ (deformation quantization) and a Hilbert
space $\mathcal{H}_{\omega}$ of states with a symplectic manifold
$(S,\omega)$ (geometric quantization), a symplectic realization should
quantize to a representation of $\mathcal{A}_{\pi}$ on $\mathcal{H}_{\omega}$,
which is the object encoding the symmetries at the quantum level (see
\cite{Xu}).
From the perspective of global integration of Poisson manifolds, symplectic
realizations naturally arise as source maps of symplectic groupoids
$(S,\omega)\rightrightarrows(M,\pi)$, which are the objects that
integrate Poisson manifolds. Symplectic realizations contain a lot
of information about integrability. For instance, Crainic and Fernandes
showed in \cite{CF2004} that a Poisson manifold is globally integrable
if and only if it admits a complete symplectic realization. However,
in the global integration case, the symplectic realization does not
carry all the information about the integrating symplectic groupoid
(some extra data is needed to define a global inverse map for instance).

The situation is different for the local/formal integration of Poisson
manifolds by local/formal symplectic groupoids. Recall that a local
symplectic groupoid is, roughly, the structure obtained by restricting
a global symplectic groupoid to a neighborhood of its unit space,
while a formal symplectic groupoid is the structure obtained by taking
$\infty$-jets of the symplectic groupoid structure maps (or their
pullbacks as in \cite{Karabegov}) at the unit space.
In the local/formal case, all the structure maps of the local/formal
symplectic groupoid, including the inverse map, can be recovered from
the source map alone (\cite{Weinstein83}), i.e., from a local/formal
symplectic realization. Another simplification in this case is that
the domain of the symplectic realization can always be taken to be
the cotangent bundle of the Poisson manifold, or, more precisely,
a local/formal neighborhood of its zero section.

Given a Poisson manifold $(M,\pi)$, the problem of constructing a
local/formal symplectic realization can essentially be tackled in
two ways, both of which start by considering a deformation $\epsilon\pi$
of the zero Poisson structure on $M$ by a parameter $\epsilon$.
For the zero Poisson structure, the canonical bundle projection $q:T^{*}M\rightarrow M$
is a symplectic realization, where the symplectic form on the cotangent
bundle is the canonical one $\omega_{0}$.

The first way is to try to deform $\omega_{0}$ into another symplectic
form $\omega_{\epsilon}$ such that the bundle projection $q$ remains
a symplectic realization from $(T^{*}M,\omega_{\epsilon})$ (or from
a neighborhood of its zero section) to $(M,\epsilon\pi)$. This approach
was the original one of Weinstein in \cite{Weinstein83}, who gave
an integral formula for $\omega_{\epsilon}$ and proved the result
in the case $M=\R^{d}$. The result for a general manifold has been
proven recently by Crainic and Marcut in \cite{CM 10}.

The second way is the approach of Karasev in \cite{Karasev} that
keeps the canonical symplectic form $\omega_{0}$ fixed and deforms
the bundle projection into a symplectic realization $q_{\epsilon}:(T^{*}M,\omega_{0})\rightarrow(M,\epsilon\pi)$.
The formal expansion of $q_{\epsilon}$ in the case $M=\R^{d}$ has
been shown in \cite{UnivGen} to coincide with the formal symplectic
realization that one can extract (see \cite{FSG}) from the tree-level
part of the Kontsevich star-product given in \cite{Kontsevich}.

The results of this paper can be summarized as follows: in the case $M=\R^{d}$, we establish the explicit relationship between the above two constructions (Proposition \ref{prop: We-Kar} and
Theorems \ref{thm: formal canonical flat realization}, \ref{thm:realizations}) and  we provide an explicit formula for the formal Karasev symplectic realization (Theorems \ref{thm:alpha_c_form} and \ref{thm:iterformula}). This explicit formula generalizes the known expression
\[q_\e(x,p) =  \sum_{n\geq 0} \frac{B_n}{n!} (-\e \ ad^*_p)^n (x), \ B_n: \ Bernoulli \ numbers,\]
valid for a linear Poisson structures on the dual of a Lie algebra $\g^*\simeq \R^q$, to the non-linear case. In the non-linear case, the underlying combinatorics is more complicated and we resort to Butcher's techniques \cite{Bu} involving rooted trees and elementary differentials. 

The outline of the paper is as follows.
In Section \ref{sec:General-results}, we review the construction
of $\omega_{\epsilon}$ as considered in \cite{CM 10,Weinstein83},
and we give a canonical way to obtain a Karasev-like realization $q_{\epsilon}$
of any Poisson manifold out of the deformed symplectic form $\omega_{\epsilon}$
by putting it in global Darboux form.

The main results are contained in Section \ref{sec:exp_form}. We prove that the formal Karasev realization on $M=\R^d$ admits the following formula 
\[ \alpha^i_\e(x,p) = x^i + \sum_{t \in [RT]} \frac{\e^{|t|}}{\sigma(t)} \ c_t \ D^i_t\overline{V}.\]
Here, $[RT]$ is the set of topological rooted trees, $c_{t}/\sigma(t)$ are coefficients generalizing the $B_n/n!$ of the linear case formula and $D_{t}^{i}\overline{V}$ is the elementary differential (see Definition \ref{def:elemdiff})
associated with $t$ and the \emph{Poisson spray} vector field $\overline{V}=-\pi^{ij}(x)p_{i}\partial_{x^{j}}$. We provide two ways of computing the coefficients: one based on the \emph{Butcher group} structure (c.f. eqs. \eqref{eq:ct} and \eqref{eq:ctrec}) and another one based on recursively defined iterated integrals: $c_t = (-1)^{|t|}I_t(0)$, with 
\[
I_{t}(\theta)=\int_{0}^{1}d\bar{\lambda}\int_{\theta}^{\bar{\lambda}}d\bar{\theta}\; I_{t_{1}}(\bar{\theta})\cdots I_{t_{m}}(\bar{\theta}), \ \ t=[t_1,..,t_m].
\]
(See Theorem \ref{thm:iterformula} for details.)

Finally, in Section \ref{sec:Comparison}  we review the formal (\emph{Kontsevich}) realization
$$s_{K}(p,x)  =  \frac{\partial S_{\frac\pi2}}{\partial p_{2}}(p,0,x),$$
where $S_{\frac\pi2}$ is the generating function (extracted from the Kontsevich
star-product) of the formal symplectic groupoid constructed in \cite{FSG} that integrates the Poisson manifold $(\R^d, \epsilon\pi)$. We prove that it coincides with the Karasev realization (and thus also with the Weinstein realization when suitably put in Darboux form) and that the underlying coefficients involving Kontsevich weights/graphs coincide with those given by the iterated integrals $I_t$ coming from the Butcher group approach (building on results of Kathotia \cite{Kathotia}).

\subsubsection*{Acknowledgments}
We thank Alberto S. Cattaneo for useful remarks and suggestions. 
B.D. acknowledges partial support FAPESP grants 2010/15069-8 and 2010/19365-0,
as well as the hospitality of the UC Berkeley mathematics department
and S\~ao Paulo University ICMC. A.C. also thanks ICMC-USP and FAPESP for support and hospitality during the elaboration of this paper. The authors also thank an anonymous referee for helping improve this paper.

\section{General results\label{sec:General-results}}

In this section, we review the symplectic realization $q:(T^{*}M,\omega_{\epsilon})\rightarrow(\R^{d},\epsilon\pi)$
considered originally by Weinstein in \cite{Weinstein83} and, more
recently, by Crainic and Marcut in \cite{CM 10}. We explain how it
can be constructed out of a Poisson spray, and we show how to put
this realization in ``global Darboux form'' to obtain a Karasev-like
realization $q_{\epsilon}:(T^{*}M,\omega_{0})\rightarrow(\R^{d},\epsilon\pi)$.
We end up by further analyzing the case $M=\R^{d}$, for which we
also consider a formal version of the constructions.

\subsection{Symplectic realizations from Poisson sprays\label{sub:Symplectic-realizations-from-sprays}}

Let $(M,\pi)$ be a Poisson manifold and $q:T^{*}M\to M$ its cotangent
bundle. Throughout the paper $V\in\mathcal{X}(T^{*}M)$ will denote
a Poisson spray for $\pi$, namely, a vector field on $T^{*}M$ satisfying:
\begin{enumerate}
\item $T_{\xi}q(V|_{\xi})=\pi^{\sharp}(\xi)$ for all $\xi\in T^{*}M$ 
\item $m_{t}^{*}V=tV$, for $m_{t}:T^{*}M\to T^{*}M$ being the diffeomorphism
obtained by fiberwise multiplication by $t\neq 0$ \end{enumerate}
\begin{rem}
Notice that $V_{\e}:=m_{\e}^{*}V=\e V$ is thus a spray for the re-scaled
Poisson structure $\e\pi$.
\end{rem}
It follows from the definition that the flow $\varphi_{s}^{V}$ of
$V$ fixes the zero section and, thus, that there exists an open neighborhood
$U_{1}\subset T^{*}M$ of the zero section such that $\varphi_{s}^{V}$
is defined for $s\in[0,1]$%
\footnote{More generally, the results below also hold for any other vector field
$V$ satisfying $(1)$ and such that there is a neighborhood of the
zero section on which the flow is defined on $[0,1]$.%
}.
\begin{example}
\label{ex: spray from nabla} A linear connection $\nabla$ on $q:T^{*}M\to M$
defines a Poisson spray by setting, for $\xi\in T^{*}M$, 
\[
V_{\xi}^{\nabla}=Hor^{\nabla}(\pi(\xi))
\]
 
\end{example}
For each $\e\in[0,1]$, we choose an open neighborhood $U_{\e}\subset T^{*}M$
of the zero section such that $\varphi_{s}^{V}$ is defined for $s\in[0,\e]$.
We can take $U_{1}\subset U_{\e}$ for all $\e\in[0,1]$. Following
\cite{CM 10}, we consider the following differential forms on $U_{\e}$

\begin{eqnarray}
\omega_{V,\e}: & = & \frac{1}{\e}\int_{0}^{\e}ds\left(\varphi_{s}^{V}\right)^{*}\omega_{0}=d\theta_{V,\e},\label{eq: def omega_V}\\
\theta_{V,\e} & := & \frac{1}{\e}\int_{0}^{\e}ds\left(\varphi_{s}^{V}\right)^{*}\theta_{0},\nonumber 
\end{eqnarray}
 where $\theta_{0}$ is the Liouville one form on $T^{*}M$ and $\omega_{0}=d\theta_{0}$
the associated symplectic form%
\footnote{Notice the sign convention $\omega_{0}\simeq dp_{i}\wedge dx^{i}=-dx^{i}\wedge dp_{i}$.%
}.

As explained in \cite{CM 10}, when evaluated at points of the zero section $0^{T^{*}M}\subset T^{*}M,$
$$\omega_{V,\e}|_{0^{T^{*}M}}=\omega_{0}-\e\pi(P_{T^*M}(\cdot),(P_{T^*M}(\cdot))$$ is non degenerate for all $\e$, where $P_{T^*M}:T|_{0^{T^*M}}(T^*M) \to T^*M$ is the natural projection. It thus
follows that we can choose possibly smaller neighborhoods $U(V)_{\e}\subset U_{\e}$
where this 2-form is non-degenerate. Then, $\omega_{V,\e}$ is a (exact)
symplectic form on this $U(V)_{\e}$.
\begin{thm}
\label{thm: omega V realization} The bundle projection $q:(U(V)_{\e}\subset T^{*}M,\omega_{V,\e})\to(M,-\e\pi)$
is a symplectic realization.
\end{thm}
This was proven in \cite{CM 10}.

\begin{rem}\label{rmk:paths}
Consider the path construction of the (local) symplectic groupoid
$(\G_{\e},\Omega_{\e})\rightrightarrows M$ integrating $(M,\e\pi)$
in terms of (small) cotangent paths modulo cotangent homotopies as
in \cite{CF2001} and \cite{CF2004}. For each initial condition, close enough to the
zero section in $T^{*}M$, the flow of $V_{\e}$ produces a particular
cotangent path $[0,1]\to T^{*}M$, defining in this way an exponential
map $\exp_{V_{\e}}:U_{1}\to\G_{\e}$. The symplectic form above arises
as $\omega_{V_{\e},1}=\exp_{V_{\e}}^{*}\Omega_{\e}=\omega_{V,\e}$.
The content of the above theorem can then be understood as follows:
the symplectomorphism $\exp_{V_{\e}}$ takes the source map of $\G_{\e}$,
which is a symplectic realization (as for any symplectic groupoid),
to the projection $q$. 
\end{rem}

\begin{rem}
\label{lem: change nabla}\emph{(Change} \emph{of spray)}. Let $(M,\pi)$
be a Poisson manifold and consider two Poisson sprays $V_{1}$ and
$V_{2}$. Then, uniqueness of symplectic realizations (see \cite{CDW1987} and the discussion below)
together with Theorem \ref{thm: omega V realization} imply the existence of a symplectomorphism
$F_{\e}:(U(V_{1})_{\e},\omega_{V_{1},\e})\to(U(V_{2})_{\e},\omega_{V_{2},\e})$
such that $ $$q\circ F_{\e}=q$.
\end{rem}

\subsection{\label{sub:Symplectic-realizations-in-Darboux}Symplectic realizations
in Darboux form}

 First, notice that the realization given by $(U(V)_{\e}\subset T^{*}M,\omega_{V,\e},q)$ is \emph{strict} in the sense of \cite{CDW1987}, since it admits a global Lagrangian section. It is shown in \cite{CDW1987} 
 that strict symplectic realizations with $1$-connected fibers are essentially unique. More precisely, symplectic realizations inherit a natural additional compatible (local) Lie groupoid structure (c.f. Remark \ref{rmk:paths}) and there is always a unique isomorphism between two such local symplectic groupoids over the same Poisson manifold (yielding an analogue of Lie's first theorem). 
 
 It then follows that two strict symplectic realizations $(S_i,\omega_i,J_i), i=1,2$ of the same Poisson manifold with $1$-connected fibers are always isomorphic: there is a symplectomorphism $\phi:(S_1,\omega_1)\to (S_2,\omega_2)$ such that \[J_2\circ \phi = J_1.\] This isomorphism, though, is not unique if one only considers the realization structure and all such $\phi$'s are parameterized by automorphisms of the $(S_i,\omega_i,J_i), i=1,2$ acting on the left and on the right, respectively. In what follows we will focus on the construction of particular isomorphisms relating symplectic realizations of certain special forms.

\begin{defn}
\label{def: psi_V global darboux} We say that a symplectic realization
$\alpha:(W,\Omega)\to(M,\pi)$ is in \emph{Darboux form} if $W\subset T^{*}M$
is an open neighborhood of the zero section and $\Omega=\omega_{0}|_{W}$.
Given a realization $r:(S,\Omega)\to(M,\pi)$, a \emph{global Darboux
frame} for it consists of a diffeomorphism $\Phi:W\subset T^{*}M\to S$
such that $\Phi^{*}\Omega=\omega_{0}$.
\end{defn}
When a global Darboux frame for $(S,\Omega)$ exists, the composition
of Poisson maps 
\[
\alpha=r\circ\Phi:(W,\omega_{0})\to(M,\pi)
\]
yields an \emph{induced realization in Darboux form}. 
By the Lagrangian embedding theorem, strict realizations always admit global Darboux frames.

Going back to our family of symplectic forms $\omega_{V,\e}$, we observe that they are exact
for all $\e$ and that $\omega_{V,\e=0}=\omega_{0}$. 
%
For completeness, we now show that, for each $\e$, there is a natural
global Darboux frame for $\omega_{V,\e}$ attached to the given data
$(M,\pi,V)$ coming from a Moser-type argument.
\begin{prop}
\label{cor: alpha for any M}Let $(M,\pi)$ be a Poisson manifold
and $V$ a Poisson spray. Consider the ($\e$-dependent) vector field
$X_{\e}^{V}$ on $U(V)_{1}$ defined by 
\begin{equation}
i_{X_{\e}^{V}}\omega_{V,\e}=-\frac{d}{d\e}\theta_{V,\e}\label{eq:symplecto}
\end{equation}
Then, there exists an open neighborhood $W\subset U(V)_{1}\subset T^{*}M$
of the zero section such that the flow $\Phi_{V,\e}$ of $X_{\e}^{V}$
is defined on $W$ for $\e\in[0,1]$ and $\Phi_{V,\e}^{*}\omega_{V,\e}=\omega_{0}$
for all $\e\in[0,1]$.\end{prop}
\begin{proof}
Notice that the forms $\omega_{V,\e}$ and $\theta_{V,\e}$ are well
defined on $U(V)_{1}$ for all $\e\in[0,1]$ and that they define
a smooth $\e$-family of differential forms on $U(V)_{1}$. The vector
field $X_{\e}^{V}$ defined in eq. \eqref{eq:symplecto} satisfies
$X_{\e}^{V}|_{0^{T^{*}M}}=0$ for all $\e$ since, when evaluated
at points of the zero section, $\theta_{V,\e}|_{0^{T^{*}M}}=0 \forall\e$.
Then, there exists an open neighborhood $W\subset U(V)_{1}$ on which
the flow $\Phi_{V,\e}$ is defined for $\e\in[0,1]$. Finally, a standard
computation shows that 
\begin{eqnarray*}
\frac{d}{d\e}\left(\Phi_{V,\e}^{*}\omega_{V,\e}\right) & = & \Phi_{V,\e}^{*}\left(L_{X_{\e}^{V}}\omega_{V,\e}\right)+\Phi_{V,\e}^{*}\left(\frac{d}{d\e}\omega_{V,\e}\right)\\
 & = & \Phi_{V,\e}^{*}\left(d\left(i_{X_{\e}^{V}}\omega_{V,\e}\right)\right)+\Phi_{V,\e}^{*}\left(d\left(\frac{d}{d\e}\theta_{V,\e}\right)\right)\\
 & = & \Phi_{V,\e}^{*}\left(d\left(i_{X_{\e}^{V}}\omega_{V,\e}+\frac{d}{d\e}\theta_{V,\e}\right)\right)\\
 & = & 0.
\end{eqnarray*}
The statement then follows from the initial condition $\omega_{V,\e=0}=\omega_{0}$.
\end{proof}
As a consequence, there exists a natural family of induced symplectic realizations
in Darboux form $q_{\e}=q\circ\Phi_{V,\e}:(W\subset T^{*}M,\omega_{0})\to(M,\e\pi)$
for $\e\in[0,1]$ associated to the spray $V$. In the particular
case $M=\R^{d}$ below, we shall see that there exists a simpler alternative construction leading
to a different Darboux realization which was introduced earlier by
Karasev (\cite{Karasev}).


\subsection{Symplectic realizations for $\R^{d}$\label{sec:The-flat-case}}

In this paragraph, we analyze further the case $M=\mathbb{R}^{n}$
together with a ``flat Poisson spray'' $V$ defined from the flat
connection $\nabla_{flat}$ on $T^{*}\R^{d}$ as in example \ref{ex: spray from nabla}.
We shall also consider its ``opposite'' $\overline{V}$ (which will
help us later on to take care of some sign issues), so that they are
given by 
\begin{equation}
V|_{\xi=(x,p)}=\pi^{ij}(x)p_{i}\partial_{x^{j}}\qquad\textrm{and}\qquad\overline{V}|_{\xi=(x,p)}=-\pi^{ij}(x)p_{i}\partial_{x^{j}}.\label{eq:flat spray V}
\end{equation}
In this case, the symplectic realization realization $q:(U(V)_{\e}\subset T^{*}\R^{n},\omega_{V,\e})\to(\R^{n},-\e\pi)$
obtained from the Poisson spray as outlined in Section \ref{sub:Symplectic-realizations-from-sprays}
coincides with the original construction by Weinstein in \cite{Weinstein83}.
We shall thus refer to it as the \emph{Weinstein realization}.
\begin{rem}
(Sign convention). Here, we will consider the ``opposite'' symplectic
realization $q:(U(\overline{V})_{\e}\subset T^{*}\R^{n},\omega_{\overline{V},\e})\to(\R^{n},\epsilon\pi)$
obtained from $\overline{V}$ in the same way. Observe that $\omega_{V,\e}=-\omega_{\overline{V},\e}$.
\end{rem}
In $\R^{d}$ there is a simple way of obtaining global Darboux frames
for $\omega_{\overline{V},\e}$ as follows. Since we are on flat space,
we can define an $\e$-family of maps 
\[
\psi_{\overline{V},\epsilon}:(x,p)\mapsto(\phi_{\epsilon}(x,p),p)
\]
from $U(V)_{\e}$ to $T^{*}\R^{d}$ by its action on coordinates 
\begin{eqnarray}
\psi_{\overline{V},\epsilon}^{*}x^{i} & = & \phi_{\e}^{i}(x,p):=\frac{1}{\e}\int_{0}^{\e}\left((\varphi_{s}^{\overline{V}})^{*}x^{i}\right)|_{(x,p)}ds\label{eq: def phi}\\
\psi_{\overline{V},\epsilon}^{*}p_{j} & = & p_{j}\nonumber 
\end{eqnarray}
It follows that 
\begin{equation}
\omega_{\overline{V},\epsilon}=\psi_{\overline{V},\epsilon}^{*}\omega_{0}=dp_{i}\wedge d\phi_{\epsilon}^{i}\label{eq: omega Rd}
\end{equation}
since the flat spray $\overline{V}$ only acts on the $x$ coordinates.

\begin{rem}
Notice that when $\pi$ is an analytic bivector on $\R^d$ then $\psi_{\overline{V},\e}$ is an analytic function of $\e$.
\end{rem}

On the other hand, since 
\begin{equation}
D_{x}\phi_{\e}(x,0)=\id\label{eq:Dphi}
\end{equation}
and $\psi_{\overline{V},\epsilon}(x,0)=(x,0)$, the inverse $\psi_{\overline{V},\epsilon}^{-1}$
exists in a neighborhood $W\subset T^{*}\R^{n}$ of the zero section.
It is then an immediate consequence of eq. \eqref{eq: omega Rd} that
$\psi_{\overline{V},\epsilon}^{-1}:W\to U(V)_{\e}$ is a global Darboux
frame for $q:(U(V)_{\e},\omega_{\overline{V},\epsilon})\to(\R^{d},\e\pi)$.
The induced Darboux-form realization is then given by the map 
\[
\alpha_{\overline{V},\epsilon}=q\circ\psi_{\overline{V},\epsilon}^{-1}:(W,\omega_{0})\rightarrow(\R^{d},\e\pi)
\]

It is not hard to see that the realization $\alpha_{\overline{V},\e}$
coincides with the one given by Karasev in \cite{Karasev} (after
rescaling the momenta using $m_\e$) and we shall refer to it as the \emph{Karasev
realization}. We have then shown:
\begin{prop}
\label{prop: We-Kar} For $M=\R^{d}$, the map $\psi_{\overline{V},\epsilon}^{-1}:W\to U(V)_{\e}$
defined above is a global Darboux frame for the Weinstein realization
defined by $\omega_{\overline{V},\e}$. Moreover, the induced Darboux-form
realization coincides with the Karasev realization $\alpha_{\overline{V},\e}$.
\end{prop}
The Karasev realization $\alpha_{\overline{V},\e}\equiv \alpha^i_\e(x,p)$ can be thus obtained from the Weinstein realization by solving the following equation
\begin{equation}
\phi_{\e}(\alpha_{\e}(x,p),p)=x\ \forall p.\label{eq: H1}
\end{equation}

\begin{rem}
Notice that the Karasev realization $\alpha_{\overline{V},\e}$ does
not coincide with the realization $q_{\e}$ coming from the general
construction of Prop. \ref{cor: alpha for any M} ($\psi_{\overline{V},\e}$
is not a flow).
\end{rem}

\begin{example}\label{ex:lin1}
 Consider a linear Poisson structure $\pi^{ij}(x)=-c^{ij}_k x^k$ in $\R^d$. It is known that $(c^{ij}_k)$ are the structure constants of some Lie algebra $\g\simeq \R^d$ so that $x$ naturally belongs to $\g^*$ and $p$ to $\g$. In this case, one has $\varphi_s^{\overline{V}}(x,p) = exp(-s \ ad^*_p) x$, where $ad^*_p\in gl(\g^*)$ denotes the coadjoint representation and $exp$ is the exponential on $gl(\R^d)$. It follows from eq. \eqref{eq: def phi} that the transformation $\psi_{\overline{V},\e}$ can be given as an analytic function
 \[ \phi_\e(x,p) = \left[ \frac{e^z -1}{z} \right]_{z=-\e \ ad^*_p} (x),\]
 for $\e, x, p$ small enough. Because of this formula, the inverse of $\phi_\e$ with respect to composition, namely the realization $\alpha_{\overline{V},\e}$ in eq. \eqref{eq: H1}, can be given in terms of the multiplicative inverse of $(e^z - 1)/z$:
 \begin{equation}
  \alpha_{\overline{V},\e}(x,p) = \left[ \frac{z}{e^z -1} \right]_{z=\e \ ad_p} (x)= \sum_{n\geq 0} \frac{B_n}{n!} (-\e \ ad^*_p)^n (x), \label{eq:lin_case}
 \end{equation}
 where $B_n$ denote the Bernoulli numbers ($B_1=-\frac{1}{2})$. Notice that this expression is clearly related to the Baker-Campbell-Hausdorff formula for the Lie algebra $\g$ (see also \cite{FSG, Kathotia}).
\end{example}

\section{Explicit formulas for the formal realization in $\R^d$}\label{sec:exp_form}

A \emph{formal symplectic realization} $s_\e:(\Ts\R^{n},\omega_{0}) \to (\R^{n},\epsilon\pi)$ is a formal power series of the form 
\[
s_{\epsilon}(p,x)=x+\epsilon s_{(1)}(p,x)+\epsilon^{2}s_{(2)}(p,x)+\cdots,
\]
where the $s_{(i)}$'s are smooth functions on $\Ts\R^{n}$, formally
satisfying (i.e. at each order in the formal parameter $\e$) the partial differential equation
\begin{equation}
\{s_{\e}^{i},s_{\e}^{j}\}_{\omega_{0}}(x,p)=\e\pi^{ij}(s_{\e}(x,p)).\label{eq: Lie system}
\end{equation}

In this section, we give an explicit formula for the formal version of the Karasev realization $\alpha_\e$, generalizing formula \eqref{eq:lin_case}  to the non-linear case. We shall express this formula in terms of rooted trees and elementary differentials, based on the foundational work of Butcher \cite{Bu}.

\subsection{Formal expansions}

We now turn to the formal setup by expanding $\omega_{\overline{V},\epsilon}$
in formal power series in $\epsilon$. The general discussion about uniqueness of realizations extends to the formal setup and we shall focus below on a particular isomorphism between formal Weinstein and Karasev realizations. Later on, in Section \textsection{\ref{sec:Comparison}}, we shall show that the formal Karasev realization coincides directly with one coming from Kontsevich's star product. 
\begin{rem}
(Conventions). We shall adopt the following conventions when working
with formal power series in $C^{\infty}(T^{*}\R^n)[[e]]$. Let $w_{\e}^{i}(x,p)=w_{0}^{i}(x,p)+\e\tilde{w}_{\e}^{i}(x,p)$
with $w_{0}^{i}(x,p)\in C^{\infty}(T^{*}\R^{n})$ and $\tilde{w}_{\e}^{i}(x,p)\in C^{\infty}(T^{*}\R^{n})[[\epsilon]]$.
For $f(x,p)\in C^{\infty}(T^{*}\R^{n})$ we shall denote (multivariable
Taylor expansion) 
\begin{equation}
f(w_{\e}(x,p),p)=\sum_{n\geq0}\frac{\e^{n}}{n!}\tilde{w}_{\e}^{i_{1}}(x,p)...\tilde{w}_{\e}^{i_{n}}(x,p)\left[\frac{\partial^{(n)}}{\partial x^{i_{1}}...\partial x^{i_{n}}}f\right](w_{0}(x,p),p).\label{eq: formal composition}
\end{equation}
 
\end{rem}
First of all notice that the formal flow of $\overline{V}$ can be
written as 
\[
\phi_{s}^{\overline{V}}(x,p)=(\exp(s\overline{V})(x),p),
\]
where $\bar{x}=\exp(s\overline{V})(x)$ is the formal flow of $\overline{V}$
(seen as a $p$-dependent vector field on $\R^{d}$) defined in components
by the following formal power series 
\[
\bar{x}^{i}=\sum_{n\geq0}\frac{s^{n}}{n!}L_{\overline{V}}^{n}(x^{i}).
\]

\begin{lem}
The map $\phi_{\e}$ defined in eq. \eqref{eq: def phi} is given
by the formal expansion 
\begin{equation}
\phi_{\e}^{i}(x,p)=\sum_{n\geq0}\frac{\e^{n}}{(n+1)!}(L_{\overline{V}})^{n}(x^{i}).\label{eq: phi}
\end{equation}
\end{lem}
This means that the formal version of the Karasev realization $\alpha_{\overline{V},\e}$
is defined by the (unique) functions $\alpha_{\e}^{i}(x,p)\in C^{\infty}(T^{*}\R^{n})[[\e]]$
solving eq. \eqref{eq: H1} with $\phi_{\e}$ given by the formal
expansion in eq. \eqref{eq: phi}.

On can actually solve for $\alpha_\e$ recursively as follows. Consider the formal expansion
\begin{equation}
\alpha_{\e}^{i}(x,p)=x^{i}+\e\sum_{r\geq1}\e^{r-1}\alpha_{(r)}^{i}(x,p).\label{eq:alpha}
\end{equation}
Using definition Equation \eqref{eq: formal composition} and Formula
\eqref{eq: phi} for $\phi_{\e}$, one gets that Equation \eqref{eq: H1} for $\alpha_\e$
is equivalent to: 
\[
x^{i}=\sum_{n,m\geq0}\sum_{r_{1},..,r_{m}\geq1}\frac{\e^{(n+r_{1}+...+r_{m})}}{(n+1)!m!}\alpha_{(r_{1})}^{i_{1}}(x,p)\dots\alpha_{(r_{m})}^{i_{m}}(x,p)\left[\frac{\partial^{(m)}(L_{\overline{V}}^{n}x^{i})}{\partial x^{i_{1}}\dots\partial x^{i_{m}}}\right](x,p).
\]
Reading this equation at order $N$ in $\e$, one gets

\begin{gather*}
\alpha_{(N)}^{i}(x,p)+\\
+\sum_{n=1}^{N-1}\sum_{m=1}^{N-n}\sum_{\underset{r+\dots+r_{m}=N-n}{r_{l}\geq1}}\frac{1}{(n+1)!m!}\alpha_{(r_{1})}^{i_{1}}(x,p)\dots\alpha_{(r_{m})}^{i_{m}}(x,p)\left[\frac{\partial^{(m)}(L_{\overline{V}}^{n}x^{i})}{\partial x^{i_{1}}\dots\partial x^{i_{m}}}\right](x,p)+
\end{gather*}
\begin{equation}
+\frac{1}{(N+1)!}(L_{\overline{V}}^{N}x^{i})(x,p)=0\label{eq: recur}
\end{equation}
This is a recursive formula for the terms $\alpha_{(N)}$
in $\alpha_{\e}$.
We summarize these results in the following 
\begin{thm}
\label{thm: formal canonical flat realization} The formal
solution $\alpha_{\e}(x,p)$ of Equation \eqref{eq: H1} is a formal
symplectic realization \textup{$\alpha_{\e}:(T^{*}\mathbb{R}^{n},\omega_{0})\to(\R^{n},\e\pi)$.}  Moreover, the realization map $\alpha_\e$ is given by the formal expansion \eqref{eq:alpha} with coefficients recursively defined by \eqref{eq: recur} and the formal Darboux
frame $\psi_{\overline{V},\e}^{-1}(x,p)=(\alpha_{\e}(x,p),p)$ takes
the formal Weinstein realization $q:(T^{*}\R^{d},\omega_{\overline{V},\epsilon})\rightarrow(\R^{d},\e\pi)$
to the formal Karasev realization given by $\alpha_{\e}$.
\end{thm}

It is easy to compute the first terms of the realization using the recursion \eqref{eq: recur}:
\begin{equation}\label{eq:firstterms}
 \alpha_\e^{i}(x,p)=x^{i} + \frac{\epsilon}{2}\pi^{vi}p_{v}+\frac{\epsilon^{2}}{12}\partial_{u}\pi^{vi}\pi^{wu}p_{v}p_{w} + \frac{\epsilon^{3}}{48}\partial_{u}\partial_{w}\pi^{vi}\pi^{ku}\pi^{lw}p_{v}p_{k}p_{w}+\mathcal{O}(\epsilon^{4}).
\end{equation}

\begin{example}
 Let us go back to the linear Poisson structure of Example \ref{ex:lin1}. In this case, $L^n_{\overline{V}}(x) = (-ad^*_p)^n(x)$, which is linear in $x$ for all $n\geq 0$. One can directly conclude that 
 \[ \alpha_{(n)}(x,p) = \frac{(-1)^n B_n}{n!} ad_p^n(x),\]
 with $B_n$ the Bernoulli numbers, solves the recursion \eqref{eq: recur}. Indeed, plugging this equation into \eqref{eq: recur}, derivatives of order $\geq 2$ vanish and we get
 \[ \left( \frac{B_N}{N!} + \sum_{k=0}^{N-1} \frac{B_k}{k! (N-k+1)!} \right) (-ad_p^*)(x) = 0,\]
 for all $N\geq 1$. The l.h.s. is known to be zero for Bernoulli numbers ($B_0 = 1$) and we thus recover \eqref{eq:lin_case} as a particular case of the general recursion \eqref{eq: recur}.
\end{example}

In the non-linear case, the underlying combinatorics needed to solve recursion \eqref{eq: recur} is more involved and we need to resort to rooted trees and elementary differentials, which we proceed to define.

\subsection{Rooted trees and elementary differentials}

A \textbf{graph} is the data $(V,E)$ of a finite set of vertices
$V=\{v_{1},\dots,v_{n}\}$ together with a set of edges $E$, which
is a subset of $V\times V$. The number of vertices is called the
\textbf{degree} of the graph and is denoted by $|\Gamma|$. We think
of $(v_{1},v_{2})\in E$ as an arrow that starts at the vertex $v_{1}$
and ends at $v_{2}$.

Two graphs are \textbf{isomorphic} if there is a bijection between
their vertices that respects theirs edges. The set $\bar{\Gamma}$
of all isomorphic graphs $ $to a given graph $\Gamma$ is called
a \textbf{topological graph}.

A \textbf{symmetry} of a graph is an automorphism of the graph (i.e.
a relabeling of its vertices that leaves the graph unchanged). The
group of symmetries of a given graph $\Gamma$ will be denoted by
$\sym(\Gamma)$. Note that the number of symmetries of all graphs
sharing the same underlying topological graph is equal; we define
the \textbf{symmetry coefficient} $\sigma(\overline{\Gamma})$ of
a topological graph $\overline{\Gamma}$ to be the number of elements
in $\sym(\Gamma)$, where $\Gamma\in\overline{\Gamma}$.

A \textbf{rooted tree} is a graph that (1) contains no cycle, (2)
has a distinguished vertex called the \textbf{root}, (3) whose set
of edges is oriented toward the root. We will denote the set of rooted
trees by $RT$ and the set of topological rooted trees by $[RT]$.

We represent graphically topological rooted trees as depicted in the
following figure:

\begin{figure}[H]
\begin{centering}
\includegraphics[scale=0.5]{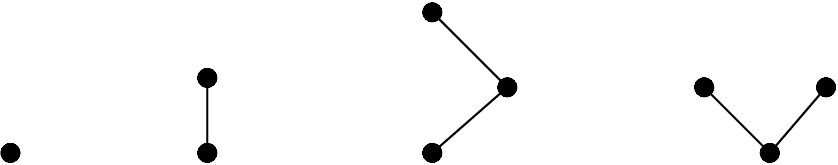} 
\par\end{centering}

\caption{\label{fig:The-rooted-tree-1}}
\end{figure}

The set of topological rooted trees can be described recursively as
follows: The single vertex graph $\bullet$ is in $[RT]$ and if $t_{1},\dots,t_{n}\in[RT]$
then so is 
\[
t=[t_{1},\dots,t_{n}]_{\bullet},
\]
where the bracket is to be thought as grafting the roots of $t_{1},\dots,t_{n}$
to a new root, which is symbolized by the subscript $\bullet$ in
the bracket $[\,,\dots,\,]$. 
\begin{example}
For instance, the formal expressions 
\[
\bullet,\qquad[\bullet]_{\bullet},\qquad[[\bullet]_{\bullet}]_{\bullet},\qquad[\bullet,\bullet]_{\bullet},
\]
correspond, from left to right, to the topological trees depicted
in Figure \ref{fig:The-rooted-tree-1}. Also observe that the total
number of ``$\bullet$'' in these formal expressions corresponds
to the degree of the rooted tree (i.e. the total number of vertices). \end{example}
\begin{rem}
Since we are dealing with topological rooted trees, the ordering in
$[t_{1},\dots,t_{m}]_{\bullet}$ is not important (for instance, we
do not distinguish between $[\bullet,[\bullet]_{\bullet}]_{\bullet}$
and $[[\bullet]_{\bullet},\bullet]_{\bullet}$).
\end{rem}

\begin{defn}\label{def:elemdiff}
Let $X=X^{i}\partial_{i}$ be a vector field on $\R^{d}$. We define
the \textbf{elementary differential} of $X$ recursively as follows:
For the single vertex tree, we define $D_{\bullet}^{u}X=X^{u}(x)$,
and for $t=[t_{1},\dots,t_{m}]$ in $[RT]$, we define 
\begin{equation}
D_{t}^{u}X=\partial_{i_{1}}\dots\partial_{i_{m}}X^{u}(x)D_{t_{1}}^{i_{1}}X\cdots D_{t_{m}}^{i_{m}}X,\label{eq:rec. form.}
\end{equation}
where we used the Einstein summation convention. For a (non-topological)
rooted tree $t\in RT$, we define $D_{t}^{u}X:=D_{\bar{t}}^{u}X$,
where $\bar{t}$ is the topological tree underlying $t$.
\end{defn}

Following \cite{Bu,HW}, one has 
\begin{equation}\label{eq:Xn}
L_X^n x^i = \sum_{t \in[ RT], \ |t|=n} \frac{|t|!}{t! \sigma(t)} D^i_t X,
\end{equation}
where the \emph{tree factorial} is recursively defined as $t!= |t| t_1! \cdots t_k!$ for $t=[t_1,..,t_k]$, $\bullet ! =1$. The symmetry factor $\sigma(t)$ also admits a recursive formula, namely, $\sigma([t_1^{n_1},..,t_1^{n_k}]) = n_1!\sigma(t_1)^{n_1} ...  n_k!\sigma(t_k)^{n_k}$ where the trees $t_1,..,t_k$ are assumed different and the exponent $n_i$ denotes it is repeated $n_i$-times.

The reader is referred to the foundational work of Butcher \cite{Bu} on elementary differentials and the use of trees in ordinary differential equations for more information.

\subsection{The Butcher group}
In a series of works culminating in \cite{Bu}, and inspired by the Runge-Kutta method for ODEs, Butcher discovered the group $G_B$ to be defined below. It tuned out to be connected to other interesting problems (see \cite{CHV} for a summary of related topics) and its appearence here is due to the fact that it encodes the combinatorics behind the construction of the Karasev realization out of the Weinstein realization via eq. \eqref{eq: H1}. Let
\[G_B = \{ a:[RT]_0 \to \R : a_\emptyset = 1 \}\]
be the set of real valued functions on rooted trees. ($[RT]_0$ is the set of topological  trees including the \emph{empty tree} $\emptyset$.) Given a vector field $X\equiv X^i(x)\partial_{x^i}$ on $\R^n$, one can associate to each such function a formal series (or $B$-series):
\[ a \mapsto B(a, \e X, x) = x + \sum_{t \in [RT]} \frac{\e^{|t|}}{\sigma(t)} \ a_t \ D_t X |_x.\]
Here $\sigma(t)$ denotes the symmetry coefficient of $t$. One can extend this assignment to $B(a, \e X, f)$ for any function $f\equiv f(x)$ by expanding formally around $\e=0$ (c.f. eq. \eqref{eq: formal composition}). 

\begin{example}
 Let $\overline{V}_p\equiv \overline{V}^i(x,p) \partial_{x^i}$ be the vector field on $\R^d$ defined by $\overline{V}$ with $p$ seen as a fixed external parameter. Then, the formal expansion eq. \eqref{eq: phi} for $\phi_\e(x,p)$ can be expressed as
 \begin{equation}\label{eq: phiB}
  \phi_\e(x,p) = B(a, \e \overline{V}_p, x), \ where \ \ a_t = \frac{1}{t!(|t|+1)}.
 \end{equation}
Above we have used formula \eqref{eq:Xn} to evaluate $L^n_{\overline{V}_p}x^i$ in terms of rooted trees. 
\end{example}

It can then be shown (\cite{Bu, HW}) that $G_B$ admits a group structure (the \emph{Butcher group} structure) $(a, b) \mapsto a\cdot b$ such that  
\[ B(b, \e X, B(a,\e X, x)) = B(a\cdot b, \e X, x).\]
The unit corresponds to $a_t = 0, \forall t \in [RT]$.

An explicit recursive formula for the product is 
\[ (a\cdot b)_t = \sum_{s\in S(t)} a(t\backslash s) b(s_t).\]
In this formula, $S(t)$ denotes the set of \emph{ordered subtrees} of $t$. These are given by subsets $s$ of vertices of $t$ (including the empty set) such that $s$ is connected with respect to the edges of $t$ and that $s$ contains the root of $t$ (when non-empty). Also above, $t\backslash s$ denotes the \emph{forest} (i.e. the disjoint union of trees) obtained from $t$ by removing the vertices of $s$ as well as their adjacent edges and the map $t \mapsto a_t$ is extended to forests as sending disjoint unions to products. Finally, $s_t$ denotes the tree with edges and root induced from those of $t$.

In what follows, we shall be mainly concerned with inversion in the Butcher group. First, notice that the above formula for the Butcher group multiplication induces a recursive formula for the inverse $a \mapsto a^{-1} \in G_B$:
\begin{equation}\label{eq:recinv}
 a^{-1}_t = - \sum_{s\in S(t), \ s\neq t} a_{t\backslash s} \ a^{-1}_{s_t}.
\end{equation}

Alternatively, there is the following explicit formula for $a \mapsto a^{-1} \in G_B$ (see \cite{CHV}):
\begin{equation}\label{eq: invB}
 a^{-1}_t = \sum_{p\in \mathcal{P}(t)} (-1)^{|p_t|} \ a_{t \backslash p}.
\end{equation}
In this formula, for each tree $t$, $\mathcal{P}(t)$ denotes the set of \emph{partitions of $t$}, namely, all subsets of the set of edges of $t$ (including the empty set). Here, $t \backslash p$ is used for the  forest obtained from $t$ after the edges in $p$ are erased while $p_t$ is the tree obtained by contracting each tree in the forest $t\backslash p$ to a single vertex and re-establishing the edges of $p$.

\subsection{An explicit formula in terms of elementary differentials}

Let $c: [RT] \to \R$ be the function defined by 
\begin{equation}\label{eq:ct}
 c_t = \sum_{p\in \mathcal{P}(t)} (-1)^{|p_t|} \ \prod_{\t \in t\backslash p} \frac{1}{\t!(|\t|+1)},
\end{equation}
where the rooted trees $\t$ are taken from the forest $t\backslash p$.   

We have the following expression for the formal Karasev realization.

\begin{thm}\label{thm:alpha_c_form}
Let $\pi$ be a Poisson structure on $\R^d$. The functions 
\begin{equation}\label{eq:alpha_el_diff}
\alpha^i_\e(x,p) = x^i + \sum_{t \in [RT]} \frac{\e^{|t|}}{\sigma(t)} \ c_t \ D^i_t\overline{V}  
\end{equation}
with coefficients $t \mapsto c_t$ given by eq. \eqref{eq:ct} define a formal symplectic realization  $\alpha_\e: (T^*R^d,\omega_0) \to (\R^d, \e\pi)$. (And this coincides with the formal Karasev realization defined by eq. \eqref{eq: H1}.)
\end{thm}

\begin{proof} 
Following eq. \eqref{eq: phiB}, we can write $\phi_\e(x,p)=B(a, \e \overline{V}_p, x)$. Let us assume that there is another element $c \in G_B$ such that $\alpha_\e(x,p) = B(c, \e \overline{V}_p, x)$. Then, the defining eq. \eqref{eq: H1} for the Karasev realization $\alpha_\e$ is equivalent to
\[x = B(c, \e \overline{V}_p, B(a, \e \overline{V}_p, x)) = B(a\cdot c,\e \overline{V}_p, x)   \ \forall x, p,\]
where $a \cdot c$ denotes multiplication in the Butcher group. Hence, $c = a^{-1}$ provides a solution to this equation. The Theorem then follows from the general formula for  the inverse in the Butcher group.
\end{proof}

One can easily compute the first coefficients:
\[
c_{\bullet}=-\frac{1}{2},\qquad c_{[\bullet]_{\bullet}}=\frac{1}{12},\qquad c_{[\bullet,\bullet]_{\bullet}}=-\frac{1}{24},\qquad c_{[[\bullet]_{\bullet}]_{\bullet}}=0,
\]
accordingly leading to the expression \eqref{eq:firstterms} for the first terms in the Karasev realization.

\begin{rem} 
One can also give a recursive formula for the coefficients $c_t$ using eq. \eqref{eq:recinv} for the inverse $a^{-1}$, namely, 
  \begin{equation}\label{eq:ctrec}
   c_t = - \sum_{s\in S(t), s\neq t} c_{s_t} \prod_{\t \in t\backslash s} \frac{1}{\t!(|\t|+1)}.
  \end{equation}
\end{rem}

\begin{example}
  Let us denote $t_N$ the \emph{tall tree} with $N$ vertices. The recursive formula \eqref{eq:ctrec} reduces to
 \[ c_{t_N} = - \sum_{k=0}^{N-1} \frac{c_{t_k}}{(N-k+1)!}.\]
 It thus follows that $c_{t_N} = \frac{B_N}{N!}, \forall N\geq 1$.
For a linear bivector $\pi$, tall trees are the only trees contributing in \eqref{eq:alpha_el_diff}, since higher derivatives of $\pi$ vanish. In this way, we recover eq. \eqref{eq:lin_case} again, now from our general formula in terms of elementary differentials.
 \end{example}

\subsection{A Formula in terms of iterated integrals} \label{subsec:iter}

In this Subsection, we show that the coefficients in the Karasev realization eq. \eqref{eq:alpha_el_diff} admit an iterated integral definition, as follows.

Let $t=[t_{1},\dots,t_{m}]$ be a topological rooted tree. We define
recursively functions of $\theta\in[0,1)$ by
\begin{equation}\label{eq:Irec}
 I_{t}(\theta)=\int_{0}^{1}d\bar{\lambda}\int_{\theta}^{\bar{\lambda}}d\bar{\theta}\; I_{t_{1}}(\bar{\theta})\cdots I_{t_{m}}(\bar{\theta}),
\end{equation}
with $I_{\bullet}(\theta)=\frac{1}{2}-\theta.$

\begin{thm}
\label{thm:iterformula}Let $\alpha_{\overline{V},\e}:(T^*\R^d, \omega_0) \to (\R^d, \e\pi)$ be the formal Karasev symplectic realization defined by the Poisson spray $\overline{V}$. Then, 
\begin{equation}
\alpha_{\overline{V},\e}^{i}(p,x)=x^i+\sum_{t\in[RT]}\frac{\e^{|t|}}{\sigma(t)}\ (-1)^{|t|} I_t(0) \ D_{t}^{i}\overline{V} .
\end{equation}
 \end{thm}

 Our proof is based on the following observation (see \cite[Section 3.4]{But10}). 
 Let $\eta$ be an $\R^n$-valued function on $[0,1]$, $f$ a vector field on $\R^n$ and $L$ a linear operator acting on scalar functions of $1$-variable. We consider the following equation
 \begin{equation}\label{eq:int_eq_eta}
 \eta^i(\xi) = y^i_0 + \e L(f^i\circ \eta)(\xi),
 \end{equation}
and express the (formal) solution in terms of elementary differentials and rooted trees. Define recursively scalar functions $\Lambda_t\equiv \Lambda_t(\xi)$
by \begin{equation}\label{eq:Lrec} \Lambda_t = L( \Lambda_{t_1}  ...  \Lambda_{t_k}) , \ \Lambda_\bullet = L( {\bf 1} ), \ \ for \ t=[t_1\dots t_k].\end{equation} 
We define successive approximations $\eta_{(k)}$ by Taylor expanding $f$ around $y_0$ up to order $k$ and using the previous approximations on the r.h.s. of the equation, namely,
\[ \eta_{(k)}^i = y_0 + \e L\left( \sum_{0\leq l \leq k} \frac{1}{l!} \eta_{(k-1)}^{i_1}\dots \eta_{(k-1)}^{i_l}
 \partial_{i_1\dots i_l}f^i (y_0)\right).
\]
By induction, one thus obtains for $k\to \infty$,
\begin{equation}\label{eq:sol_eta}
 \eta^i(\xi) = y^i_0 + \sum_{t\in [RT]} \frac{\e^{|t|}}{\sigma(t)}\ \Lambda_t(\xi) \ D^i_tf(y_0).
\end{equation}

\begin{proof}(of Thm. \ref{thm:iterformula}) 
Let $f$ be an arbitrary vector field on $\R^n$, with (formal) flow $\varphi_s=(\varphi_s^1,\dots, \varphi_s^n)$. Consider the associated function \[\phi_\e(y) = \frac{1}{\e} \int_0^\e \varphi_s(y) \ ds, \ y\in \R^n,\] and its inverse $\alpha_\e \equiv \alpha_\e (y)$ defined by \[ \phi_\e (\alpha_\e(y))= y \ \forall y.\]
We know that $\phi_\e(y) = B(a, \e f, y)$ is given by a B-series with $a\in G_B$ defined in eq. \eqref{eq: phiB} and we look for $c\in G_B$ such that $\alpha_\e(y)=B(c,\e f, y)$.

It follows from the definition of $\phi_\e$ that
\[ \phi_\e( \varphi_{\e \xi}(y)) = \phi_\e(y) + \e \int_0^\xi du \int_u^{u+1}  ds \ f(\varphi_{\e s}(y)).\]
 On the other hand, $\phi_\e(y) = y + \e \int_0^1du \int_0^u ds \ f(\varphi_{\e s}(y))$. It thus follows that,
 \begin{eqnarray*} 
\phi_\e(\varphi_{\e\xi}(\alpha_\e(y)))&=& \phi_\e(\alpha_\e(y)) +  \e \int_0^\xi du \int_u^{u+1}  ds \ f(\varphi_{\e s}(\alpha_\e(y))), \\
\varphi_{\e\xi}(\alpha_\e(y)) + \e \int_0^1du \int_0^u ds f(\varphi_{\e s + \e \xi}(\alpha_\e(y))) &=&
y +  \e \int_0^\xi du \int_u^{u+1}  ds \ f(\varphi_{\e s}(\alpha_\e(y))).
\end{eqnarray*}
 where, on the r.h.s. of the second step we used that $\alpha_\e$ is the inverse of $\phi_\e$.
 
 Denoting $\eta^i(\xi) = \varphi^i_{\e \xi}(\alpha_\e(y))$ and defining the operator 
 \[ g(\xi) \mapsto L(g)(\xi) = \int_0^\xi du  \int_u^{u+1}ds \ g(s)-\int_0^1 du \int_\xi^{u+\xi}ds \ g(s),\]
 it follows that $\eta$ must be a solution of eq. \eqref{eq:int_eq_eta}. From the expression \eqref{eq:sol_eta} for the solution we get
 \[ \alpha^i_\e(y) = \eta^i(0) = \sum_{t\in RT}  \frac{\e^{|t|}}{\sigma(t)}\ \Lambda_t(0) \ D^i_tf(y_0),\]
 where the iterated integrals $\Lambda_t(\xi)$ are defined by the recursion \eqref{eq:Lrec}. We have thus shown that $c = \Lambda_t(0)$ yields the B-series for $\alpha_\e$ for any $f$. Finally, one can verify that
 \[ L(g)(\xi) = -\int_0^1 du \int_\xi^u ds \ g(s) ,\]
 by examining the underlying $2d$ (polygonal) region of integration in the first definition of $L$, noting that the integrand only depends on $s$ (not on $u$) and thus deforming this region along the $u$-axis without changing the value of the integral. Recalling the definition \eqref{eq:Irec} of the iterated integrals $I_t$, we get $I_t = (-1)^{|t|} \Lambda_t$ and the result thus follows.
\end{proof}

The iterated integrals \eqref{eq:Irec} are very easy to compute; for
the first rooted trees, we have 
\begin{eqnarray*}
I_{\bullet}(\theta) & = & -\theta+\frac{1}{2},\\
I_{[\bullet]_{\bullet}}(\theta) & = & \frac{\theta^{2}}{2}-\frac{\theta}{2}+\frac{1}{12},\\
I_{[\bullet,\bullet]_{\bullet}}(\theta) & = & -\frac{\theta^{3}}{3}+\frac{\theta^{2}}{2}-\frac{\theta}{4}+\frac{1}{24},\\
I_{[[\bullet]_{\bullet}]_{\bullet}}(\theta) & = & -\frac{\theta^{3}}{6}+\frac{\theta^{2}}{4}-\frac{\theta}{12}.
\end{eqnarray*}
With these, we recover once more formula \eqref{eq:firstterms} for the first terms of the formal Karasev realization.

\section{Comparison to the Kontsevich realization\label{sec:Comparison}}

In this section, we recall the formal \emph{Kontsevich realization} $s_K:(T^*\R^d, \omega_0) \to (\R^d, \e\pi)$ introduced in \cite{FSG}
from the (tree-level part of) Kontsevich's quantization formula. We show that it coincides with the formal Karasev realization and, moreover, that the only contributing Kontsevich weights are exactly  those defined by the recursion \eqref{eq:Irec}.
%
%
%
%
%
%
%

\subsection{The generating function and the Kontsevich realization}

In \cite{FSG}, it was shown that the formal power series 
\begin{equation}
S_{\frac{\pi}{2}}(p_{1},p_{2},x)=(p_{1}+p_{2})x+\sum_{n=1}^{\infty}\frac{\epsilon^{n}}{n!}\sum_{\Gamma\in T_{n,2}}W_{\Gamma}\hat{B}_{\Gamma}\left(\frac{\pi}{2}\right)(p_{1},p_{2},x),\label{eq:gen. fct.}
\end{equation}
with $x\in\R^{d}$ and $p_{1},p_{2}\in(\R^{d})^{*}$ is a formal generating
function for the formal symplectic groupoid $\Ts\R^{n}\rightrightarrows\R^{d}$
(where the cotangent bundle is endowed with its canonical symplectic
structure $\omega_{0}=\sum_{i}dp^{i}\wedge dx^{i}$) integrating $(\R^{n},\epsilon\pi)$.
In the above formula, 
\begin{itemize}
\item $T_{n,2}$ is the set of \textbf{Kontsevich trees} of type $(n,2)$; 
\item $W_{\Gamma}$ is a real number called the \textbf{Kontsevich weight}
of $\Gamma$; 
\item $\hat{B}_{\Gamma}(\pi)$ is the symbol of the \textbf{Kontsevich operator}
$B_{\Gamma}(\pi)$. 
\end{itemize}
We refer the reader to Appendix \ref{app:Kgraphs} 
for the definitions of these objects and to \cite{Kontsevich,UnivGen} for more detail. 

As shown in \cite{FSG}, the formal source map can be
extracted from the formal generating function as follows: 
\begin{eqnarray}
s_{K}(p,x) & = & \frac{\partial S_{\frac{\pi}{2}}}{\partial p_{2}}(p,0,x).\label{eq:source}
\end{eqnarray}
Explicitly, the formal realization $s_K:(\Ts\R^{n},\omega_{0})\to (\R^{n},\epsilon\pi)$ is given by 
\begin{equation}
s_{K}(p,x)=x+\sum_{n=1}^{\infty}\frac{\epsilon^{n}}{n!}\sum_{\Gamma\in T_{n,2}}\frac{W_{\Gamma}}{2^{n}}\frac{\partial\hat{B}_{\Gamma}(\pi)}{\partial p_{2}}(p,0,x).\label{eq: K-source}
\end{equation}
and we will refer to it as the \emph{formal Kontsevich realization}.

\begin{rem} (Conventions)
The formal generating function above depends on the Poisson structure
through the symbols of the Kontsevich operators. In \cite{FSG} and
\cite{UnivGen}, the scaling of the Poisson structure is different:
$S_{\pi}$ is used instead of $S_{\frac{\pi}{2}}$, and thus the corresponding
formal symplectic groupoid integrates the Poisson structure $2\epsilon\pi$
instead of $\epsilon\pi$. We also use Kathotia's convention (\cite{Kathotia})
for the Kontsevich weights; namely, 
\[
W_{\Gamma}=n!W_{\Gamma}^{K},
\]
where $W_{\Gamma}^{K}$ is the weight actually defined by Kontsevich
in \cite{Kontsevich}.
Also, 
the factor $\frac{1}{2^{n}}$ in eq. \eqref{eq: K-source}
comes from the recalling of the Poisson structure: namely, 
\[
B_{\Gamma}(\frac{1}{2}\pi)=\frac{1}{2^{n}}B_{\Gamma}(\pi),\quad\textrm{for }\Gamma\in T_{n,2}.
\]
\end{rem}

\subsection{Comparison between Karasev and Kontsevich realizations}

As mentioned in the Introduction, the realization $\alpha_{\overline{V},\epsilon}$ is a formal version
of the one introduced by Karasev in \cite{Karasev}. This, in turn,
was shown in \cite{UnivGen} to coincide with the realization $s_{K}$
coming from Kontsevich trees. Here we explain the argument for completeness (c.f. Theorem \ref{thm:realizations}).

The key tool is the following Lemma:
\begin{lem}(\cite{UnivGen})
\label{lem:uniqueness flat realizations} All formal symplectic realizations
from $(\Ts\R^{n},\omega_{0})$ to $(\R^{n},\e\pi)$ starting with
the same first order term $s_{(1)}$ and such that, for $i\geq1$,
we have 
\begin{eqnarray}
\langle s_{(i)}(p,x),p\rangle & = & 0,\quad i\geq1,\label{eq:vanish}\\
s_{(i)}(\lambda p,x) & = & \lambda^{i}s_{(i)}(p,x),\label{eq:homog}
\end{eqnarray}
 coincide.\end{lem}

As shown in \cite{UnivGen}, $s_{K}$ satisfies the hypothesis of
this Lemma. As for $\alpha_{\overline{V},\epsilon}$, writing it as \eqref{eq:alpha} and using equation \eqref{eq: recur}, we get that the first term of $\alpha_{\overline{V},\epsilon}$ is $\frac{1}{2}\pi^{vi}p_{v}$,
which is the same as for the Kontsevich realization. By inspection
of the recursive formula \eqref{eq: recur}, we see that each term of $\langle\alpha_{(N)}(x,p),p\rangle$
has a factor of the form 
\[
\cdots\partial\cdots\partial L_{\overline{V}}(x^{i})p_{i}=\cdots\partial\cdots\partial\pi^{vi}p_{v}p_{i}
\]
which is zero by the antisymmetry of the Poisson structure. This shows
that \eqref{eq:vanish} holds. An induction on $N$ using the same
equation immediately yields \eqref{eq:homog}. 

\begin{thm}
\label{thm:realizations}Let $\alpha_{\overline{V},\e}:(T^{*}\R^{n},\omega_{0})\to(\R^{n},\e\pi)$
be the formal Karasev realization defined by the spray $\overline{V}$ in eq. \eqref{eq: H1}. 
This realization coincides with the Kontsevich realization $s_{K}$.
\end{thm}

%

%
%

\subsection{Kontsevich graphs from trees and recursive formula for the weights}

Not all of the graphs in $T_{n,2}$ contribute to $s_K$. Here we identify which of them do and compute the associated weights, re-obtaining the iterated integrals \eqref{eq:Irec} from a different perspective.

To each rooted tree $t\in RT$ one can associate $\Gamma_{t}\in T_{n,2}$ being the Kontsevich tree whose interior is $t$ and such that all of its terrestrial edges land on $\bar{1}$ except for the root of
$t$, which has its first edge landing on $\bar{1}$ and its second
landing on $\bar{2}$. All the edges landing on $\bar{1}$ are considered
to be the first ones. (See Appendix \ref{app:Kgraphs} for details.)
It is not hard to check that the above assignment descends to topological trees so that
\[
W_{\bar{t}}:=W_{\Gamma_{t}},\quad t\in\bar{t}, \ \bar{t}\in [RT],
\]
is well defined. 
We have thus defined Kontsevich weights for (topological) rooted trees which we now proceed to compute.

\begin{prop}
\label{prop:weights}Let $t\in[RT]$ and $W_{t}$ be the corresponding
Kontsevich weight. Then 
\[
W_{t}=I_{t}(0),
\]
where the functions $I_{t}(\theta)$ were recursively defined in eq. \eqref{eq:Irec}.
 \end{prop}
\begin{proof}
We proceed by induction on the tree degree, following the strategy
outlined by Kathotia in \cite{Kathotia}. For one vertex, we have
$I_{\bullet}(0)=\frac{1}{2}$, which is the Kontsevich weight of $\Gamma_{\bullet}$.
Suppose this is true for topological rooted trees with $n$ vertices.
Let $t=[t_{1,}\dots,t_{m}]$ such that $|t_{i}|<n$ for $i=1,\dots,m$.
Consider the graph $\Gamma_{t}(\theta)$ made out of $\Gamma_{t}$
by letting the root second edge point in the air toward an additional
fixed vertex $z$ at distance $1$ and angle $\theta$ from $\bar{1}$
as in the following figure:

\begin{figure}[H]
\begin{centering}
\includegraphics[scale=0.3]{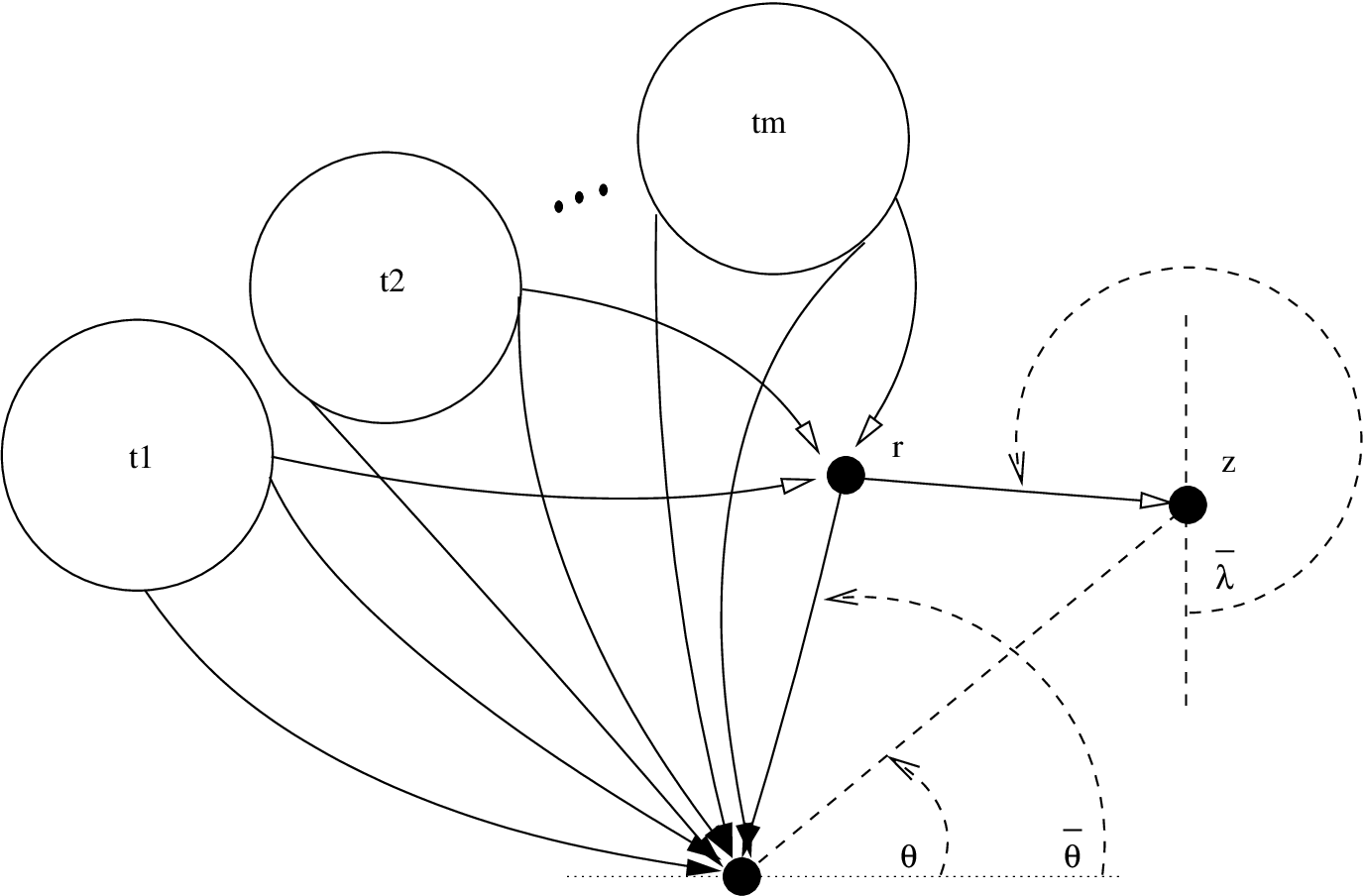} 
\par\end{centering}

\caption{\label{fig:Formual}}
\end{figure}

Let $\omega_{\Gamma_{t}(\theta)}$ be the angle-form defined by the
angle subtended by the aerial vertices of $\Gamma_{t}(\theta)$ minus
the additional vertex $z$. We define 
\[
W_{\Gamma_{t}}(\theta)=\int_{\bar{M}_{|t|}}\omega_{\Gamma_{t}(\theta)},
\]
where the integration is over the configuration space of the $|t|$
aerial vertices of $\Gamma_{t}$. Clearly, $W_{t}=W_{\Gamma_{t}(\theta)}(0)$.
Denote by $\bar{\theta}$ and $\bar{\lambda}$ the angle subtended
by the root $r$ of $t$ in $\Gamma_{t}(\theta)$ at $\bar{1}$ and
$z$ as in Figure \ref{fig:Formual}. By Fubini, we have that 
\[
W_{\Gamma_{t}}(\theta)=\int_{0}^{1}d\bar{\lambda}\int_{\theta}^{\bar{\lambda}}d\bar{\theta}\;\int_{\overline{M}_{|t|-1}}\omega_{\Gamma_{t}(\theta)-z},
\]
where $\omega_{\Gamma_{t}(\theta)-z}$ is the angle-form subtended
by the aerial vertices of $\Gamma_{t}(\theta)$ minus $z$. Because
there is no aerial arrow between the aerial vertices of $t_{i}$ and
$t_{j}$ if $i\neq j$, we have that 
\begin{eqnarray*}
(\omega_{\Gamma_{t}(\theta)-z})_{|r=(\bar{\theta},\bar{\lambda})} & = & \omega_{\Gamma_{t_{1}}(\bar{\theta})}\wedge\cdots\wedge\omega_{\Gamma_{t_{m}}(\bar{\theta})},
\end{eqnarray*}
and by Fubini, we obtain 
\begin{eqnarray*}
W_{\Gamma_{t}}(\theta) & = & \int_{0}^{1}d\bar{\lambda}\int_{\theta}^{\bar{\lambda}}d\bar{\theta}\;\left(\int_{\overline{M}_{|t_{1}|}}\omega_{\Gamma_{t_{1}}(\bar{\theta})}\right)\cdots\left(\int_{\overline{M}_{t_{m}}}\omega_{\Gamma_{t_{m}}(\theta)}\right),\\
 & = & \int_{0}^{1}d\bar{\lambda}\int_{\theta}^{\bar{\lambda}}d\bar{\theta}\; W_{\Gamma_{t_{1}}}(\bar{\theta})\cdots W_{\Gamma_{t_{m}}}(\bar{\theta}).
\end{eqnarray*}
In \cite{Kathotia}, Kathotia computed that $W_{\Gamma_{\bullet}}(\theta)=\frac{1}{2}-\theta$,
which allows us to conclude that $I_{t}(\theta)=W_{\Gamma_{t}}(\theta)$
and $I_{t}(0)=W_{\Gamma_{t}}$. 
\end{proof}

We thus get as a corollary of Thm. \ref{thm:iterformula} that the only graphs contributing to $s_K$ are those of the form $\Gamma_t$ for $t$ a rooted tree:
\begin{cor}
\label{thm:sKformula} $s_{K}$ is given by 
 \begin{equation}
 s_{K}^{i}(p,x)=x+\sum_{t\in[RT]}\frac{\e^{|t|}}{\sigma(t)} \ (-1)^{|t|} W_{t} \ D_{t}^{i}\overline{V}.\label{eq:KK-source}
 \end{equation}
 \end{cor}



%
%

\begin{appendix}

\section{Kontsevich's graphs, operators and weights}\label{app:Kgraphs}

A \textbf{Kontsevich graph} of type $(n,m)$ is a graph $(V,E)$ whose
vertex set is partitioned in two sets of vertices $V=V^{a}\sqcup V^{g}$,
the \textbf{aerial vertices} $V^{a}=\{1,\dots,n\}$ and the \textbf{terrestrial
vertices} $V^{g}=\{\bar{1},\dots,\bar{m}\}$ such that 
\begin{itemize}
\item all edges start from the set $V^{a}$, 
\item loops are not allowed, 
\item there are exactly two edges going out of a given vertex $k\in V^{a}$, 
\item the two edges going out of $k\in V^{a}$ are ordered, the first one
being denoted by $e_{k}^{1}$ and the second one by $e_{k}^{2}$. 
\end{itemize}
An \textbf{aerial edge} is an edge whose end vertex is aerial, and
a \textbf{terrestrial edge} is an edge whose end vertex is terrestrial.
We denote by $G_{n,m}$ the set of Kontsevich graphs of type $(n,m)$.
Figure \ref{Fig: K-2-3} illustrates a graphical way to represent
Kontsevich graphs.

Given a Poisson structure $\pi$ and a Kontsevich graph $\Gamma\in G_{n,m}$,
one can associate a $m$-differential operator $B_{\Gamma}(\pi)$
on $\R^{n}$ in the following way: For $f_{1},\dots,f_{m}\in C^{\infty}(\R^{n})$,
we define 
\[
B_{\Gamma}(\pi)(f_{1}\dots,f_{m}):=\sum_{I:E_{\Gamma}\rightarrow\{1,\dots,d\}}\big[\prod_{k\in V_{\Gamma}^{a}}(\prod_{\substack{e\in E_{\Gamma}\\
e=(*,k)
}
}\partial_{I(e)})\pi^{I(e_{k}^{1})I(e_{k}^{2})}\big]\prod_{i\in V_{\Gamma}^{g}}\big(\prod_{\substack{e\in E_{\Gamma}\\
e=(*,i)
}
}\partial_{I(e)}\big)f_{i}.
\]
 The symbol $\hat{B}_{\Gamma}$ of $B_{\Gamma}$ is defined by 
\[
B_{\Gamma}(e^{p_{1}x},\dots,e^{p_{m}x})=\hat{B}_{\Gamma}(p_{1},\dots,p_{m})e^{(p_{1}+\cdot+p_{m})x}.
\]
 The following figure illustrates this.

\begin{figure}[H]
\begin{centering}
\includegraphics[scale=0.4]{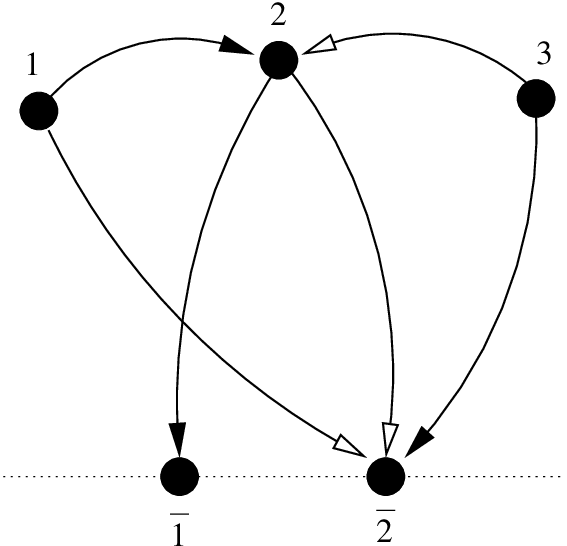} 
\par\end{centering}

\centering{}\caption{\label{Fig: K-2-3}A Kontsevich graph $\Gamma$ of type $(n,2)$.
The terrestrial vertices are on an imaginary line (the dotted line)
and the aerial vertices are placed above them. The first arrow stemming
out of an aerial vertex has a solid black head, while the second one
has a hollow white one. }
\end{figure}

\begin{example}
Consider the Kontsevich graph $\Gamma$ as in Figure \ref{Fig: K-2-3}.
The three aerial vertices $1,2$ and $3$ correspond to three copies
of the Poisson structure, say $\pi^{ij}$, $\pi^{kl}$ and $\pi^{mn}$.
The two terrestrial vertices $\bar{1}$ and $\bar{2}$ correspond
to two functions smooth functions on $\R^{d}$, say $f$ and $g$.
The arrows correspond to derivatives. Therefore, the Kontsevich operators
associated with this graph is 
\[
B_{\Gamma}(\pi)(f,g)=\pi^{ij}\partial_{n}\partial_{i}\pi^{kl}\pi^{mn}\partial_{k}f\partial_{m}\partial_{l}\partial_{j}g,
\]
where we use the Einstein summation convention of repeated indices.
The corresponding symbol is obtained by replacing $\partial_{i}f$
by $p_{i}^{1}$ and $\partial_{j}g$ by $p_{j}^{2}$: namely, 
\[
\hat{B}_{\Gamma}(\pi)(p_{1},p_{2},x)=\pi^{ij}\partial_{n}\partial_{i}\pi^{kl}\pi^{mn}p_{k}^{1}p_{m}^{2}p_{l}^{2}p_{j}^{2}.
\]
The order of the arrow is important because flipping, for example,
the order of the first aerial vertex order would introduce a sign,
since $\pi^{ij}=-\pi^{ji}$. 
\end{example}
$T_{n,2}$ is a subset of $G_{n,2}$ that we now define:
\begin{defn}
Let $\Gamma\in G_{m,n}$ be a Kontsevich graph. The \textbf{interior}
of $\Gamma$ is the graph $\Gamma_{i}$ obtained from $\Gamma$ by
removing all terrestrial vertices and terrestrial edges. A Kontsevich
graph is a \textbf{Kontsevich tree} if its interior is a tree in the
usual sense (i.e. it has no cycles). We denote by $T_{n,m}$ the set
of Kontsevich's trees of type $(n,m)$. 
\end{defn}
We give now a very rough presentation of the Kontsevich weights, which
follows Kathotia's conventions in \cite{Kathotia}, instead of the
original ones of Kontsevich in \cite{Kontsevich}. Let $\mathcal{H}$
be the upper half plane and $\R_{x}$ be the $x$-coordinate line
in $\mathcal{H}$. Given two points $p$ and $q$ in $\mathcal{H}$
with $p\notin\R_{x}$ , we define the \textbf{angle function} $\phi(p,q)\in[0,1)$,
which depends on whether $q\in\R_{x}$ or not ($\phi(p,q)=\theta$
if $q\in\R_{x}$ and $\phi(p,q)=\lambda$ if $q\notin\R_{x}$), as
depicted in Figure \ref{fig:angles} (we refer the reader to \cite{Kathotia,Kontsevich}
for more details on the angles):

\begin{figure}[H]
\begin{centering}
\includegraphics[scale=0.5]{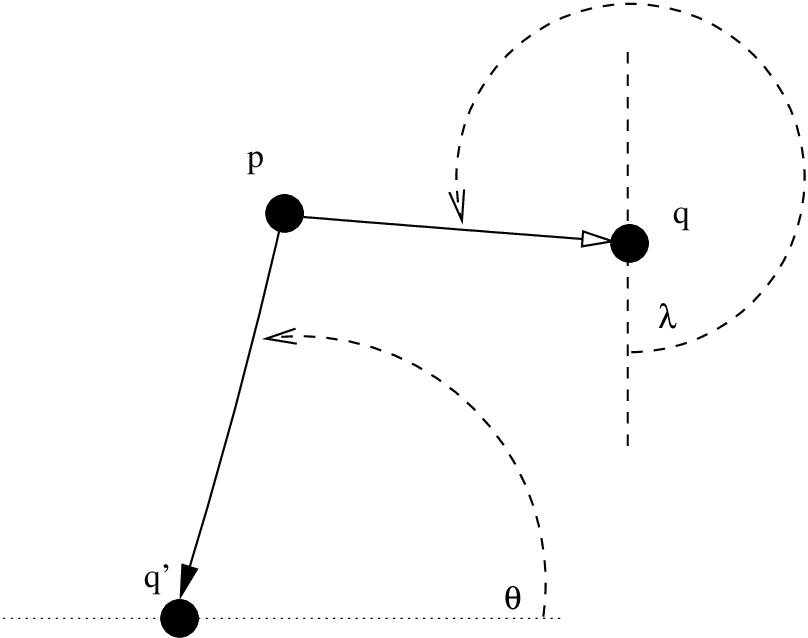} 
\par\end{centering}

\caption{\label{fig:angles}}
\end{figure}

Observe that, for fixed $q$ and $q'$ as in Figure \ref{fig:angles},
the angles $\theta$ and $\lambda$ determine completely the position
of the point $p$.

A \textbf{configuration} $z=(z_{1},\cdots,z_{n})$ with $z_{i}\in\mathcal{H}$
of a Kontsevich's graph $\Gamma\in G_{n,2}$ is an identification
$i\mapsto z_{i}$ of its aerial vertices $\{1,\dots,n\}$ with the
points of the $n$-tuple $z$ and of its terrestrial vertices $\bar{1}$
and $\bar{2}$ with respectively $0$ and $1$. We denote by $M_{n}$
the manifold of all the configurations of $\Gamma$; its compactification
$\bar{M}_{n}$ is a manifold with corners (see \cite{Kontsevich}
for details).

Given $\Gamma\in G_{n,2}$, we have $n$ pairs of angle-functions
on $\bar{M}_{n}$, which we denote by $\phi_{1}^{1},\phi_{1}^{2},\dots,\phi_{n}^{1},\phi_{n}^{2}$.
They are determined by the configuration of the aerial vertices of
$\Gamma$: Namely, $\phi_{k}^{i}=\phi(z_{k},z_{\gamma^{i}(k)})$ is
the angle-function as above, where $\gamma^{i}(k)$ is the vertex
target of $e_{k}^{i}$ (the $i^{th}$ edge of $k$). This allows us
to define the angle-form on $\bar{M}_{n}$ associated with $\Gamma$:
\[
\omega_{\Gamma}=\bigwedge_{k=1}^{n}d\phi_{k}^{1}\wedge d\phi_{k}^{2}.
\]
Since $\omega_{\Gamma}$ is a $2n$-form and $\bar{M}_{n}$ is a compact
real $2n$-manifold, we can integrate $\omega_{\Gamma}$, which give
us the Kontsevich weight of $\Gamma$: 
\[
W_{\Gamma}=\int_{\bar{M}_{n}}\omega_{\Gamma}.
\]

\begin{rem}
Following Kathotia, we do not include the original factor $\frac{1}{n!}$
in the definition of the weights, and the factor $\frac{1}{(2\pi)^{n}}$
is taken into account by the rescaling of the angles (we use $[0,1)$
instead of $[0,2\pi)$).
\end{rem}

\end{appendix}

\end{document}